\RequirePackage{fix-cm}
\documentclass[smallextended]{svjour3}       
\smartqed  
\usepackage{graphicx}
\newcommand{\ii}{\'\i}
\begin{document}

\title{Variational characterization of the speed of reaction diffusion fronts for gradient dependent diffusion}

\titlerunning{Variational Characterization for the  speed of fronts}        

\author{Rafael~D.~Benguria \and
        M.~Cristina~Depassier
}


\institute{R.~D.~Benguria \at
              Instituto de F\ii sica, Pontificia Universidad Cat\'olica de Chile, Casilla 306, Santiago 22, Chile \\
             \email{rbenguri@uc.cl}           
           \and
            M.~C.~Depassier \at
            Instituto de F\ii sica, Pontificia Universidad Cat\'olica de Chile, Casilla 306, Santiago 22, Chile \\
             \email{mcdepass@uc.cl}    
}

\date{Received: date / Accepted: date}

\maketitle

\begin{abstract}
We study the asymptotic speed of  travelling fronts of the scalar reaction diffusion  for positive reaction terms and  with a diffusion 
coefficient depending nonlinearly  on the  concentration and on  its gradient. We restrict our study to diffusion coefficients of the form 
$D(u,u_x) = m u^{m-1} u_x^{m(p-2)}$  for which existence and convergence to travelling fronts has been established.   We formulate a 
variational principle for the asymptotic speed of the fronts.  Upper and lower bounds for the speed valid for any $m\ge0, p\ge 1$ are 
constructed.  When $m=1, p=2$ the problem reduces to the  constant diffusion problem and the bounds correspond to   the classic 
Zeldovich--Frank--Kamenetskii lower bound and the Aronson-Weinberger upper bound respectively.  In the special case  
$m(p-1) = 1$ a local lower bound can be constructed  which coincides with the aforementioned upper bound. The speed in this case is 
completely determined  in agreement with recent results. 



\subclass{MSC 35 K 57 \and MSC 35 K 65 \and MSC 35 C 07 \and  MSC 35 K 55 \and MSC 58 E 30}
\end{abstract}

\section{\, Introduction}
\label{intro}
In this work we study the asymptotic propagation of  fronts of the scalar reaction diffusion equation, 
\begin{equation}\label{partial}
\partial_t u = \partial_x ( |\partial_x u^m|^{p-2} \partial_x u^m) + f(u),  \qquad f(0)=f(1)=0, f(u)>0 \qquad \mbox{in} \qquad  (0,1),
\end{equation}
which  reduces to the classical problem \cite{KPP37} when $m=1, p=2$. The diffusion term can be seen either  as the scalar version of 
the p-Laplacian acting on $u^m$  or as reaction diffusion equation with nonlinear diffusion coefficient 
$D(u,u_x) = m u^{m-1} |u_x|^{m(p-2)}$. Such diffusion coefficients are encountered, for example,  in hot plasmas 
\cite{Jardin2008,Wilhelm2001} and  the corresponding processes are referred to as doubly nonlinear diffusion  processes 
\cite{Vasquez}.

The classical problem $m=1$, $p=2$,  is fully understood \cite{AW78,KPP37}.  When nonlinear diffusion is included several scenarios 
may arise depending on the precise form of the 
diffusion coefficient. The case of  a power of concentration diffusion coefficient of the form $D(u) =u^s$  has been studied extensively 
beginning with the analytical solution found for 
$s=2$. Existence and convergence results are known for all $s$. A distinctive feature of  density dependent diffusion is the appearance 
of a finite wave at the asymptotic speed.
This is true even in the simpler $p=2$ case when $m>1$ (see, e.g., \cite{Ar80,BeDe95}, and references therein).
In recent work   \cite{Vasquez} the more general case of doubly nonlinear diffusion is considered. It is shown that for all $m>0, p>1$  
such that $\gamma = m (p-1)-1 >0$,  a unique 
monotonic increasing travelling wave joining the equilibria $u=0$ and $u=1$ exists for speeds $c\ge c_*(m,p)$ and none if 
$0<c<c_*(m,p)$.  
For $c=  c_*(m,p)$ the travelling wave (TW) is finite, whereas for $c>  c_*(m,p)$ 
the TW is positive (see \cite{Vasquez}, Theorem 2.1).  
In the case $\gamma=0$ a unique monotonic increasing travelling wave joining the equilibria $u=0$ and $u=1$ exists for speeds $c\ge 
c_*(m,p)$ and none if $0<c<c_*(m,p)$.  
For $c=  c_*(m,p)$ the travelling wave (TW) is positive  (see \cite{Vasquez}, Theorem 2.2).  Moreover, when $\gamma=0$, 
an explicit expression for the minimal speed is given, 
\begin{equation}\label{c0}
c_*(m,p)|_{\gamma=0} \equiv c_0(m,p) = ( m^2 p^{m+1} f'(0))^{1/(m+1)}.
\end{equation}
The convergence of suitable initial conditions to the travelling wave  of minimal speed  is  demonstrated in \cite{Vasquez} as well. 

The purpose of this work is to establish a variational characterization for the speed $c_*(m,p)$.  The exact value of the speed cannot 
be determined in general  however upper and lower bounds on the speed for general values of $m$ and  $p$ can be obtained.
The main result of the present work (see Theorem 1 below) is the variational  expression for the speed
\begin{equation}\label{vpIntro}
c_* = \sup_g  \left[ p \left( \frac{m}{p-1}\right)^{(p-1)/p}  \frac{\int_0^1  u^{(m-1)(p-1)/p}  h^{1/p}  f^{(p-1)/p}  g^{(p-1)/p}  du}{\int_0^1 g(u) 
du} \right],
\end{equation}
from where upper and lower  bounds will be constructed. In (\ref{vpIntro}), $g \in C^1(0,1)$ is such that $g(u) \ge 0$, with 
$h(u) \equiv - g'(u) >0$ in $(0,1)$ and $\int_0^1 g(u) \, du$ finite.

We find that for any $m>0,p>1$ (with $\gamma \ge 0$) the asymptotic speed 
is bounded  by
\begin{eqnarray}
  & \left( \frac{m p}{p-1} \int_0^1 u^{(m-1)} f(u) du\right)^{(p-1)/p} 
   \le c_*(m,p) \le \nonumber \\
   & p \left( \frac{m}{p-1}\right)^{ (p-1)/p} \sup_u \left[ u^{\gamma}  \left( \frac{f}{u}\right)^{(p-1)}\right]^{1/p}.
\label{bounds}
\end{eqnarray}
The lower bound is a generalization of the Zeldovich-Frank-Kamenetskii (ZFK) bound (see, e.g., \cite{BeNi92,BeDe98}).  Effectively, 
for $m=1, p=2$ their classical bound 
$$
c_* \ge c_{ZFK} = \sqrt{ 2 \int_0^1 f(u) du}
$$ is recovered.
 The upper bound, for $m=1,p=2$ reduces to the Aronson-Weinberger  upper bound \cite{AW78}
 $$
 c\le \sup_u 2 \sqrt{ \frac{f(u)}{u}} \, .
 $$
 An interesting case arises when $\gamma=0$. As mentioned above  the speed can be  determined exactly  \cite{Vasquez} and it is 
 given by
 $c_0(m,p)$ when $\gamma=0$. Here we recover this result from the variational principle showing that when $\gamma=0$ a local 
 lower bound can be found choosing an adequate trial function $g(u)$.  This lower bound is exactly $c_0(m,p)$. The upper bound 
 given in (\ref{bounds}) reduces to $c_0(m,p)$ when $\gamma =0$ and $f(u)$ satisfies the KPP criterion $sup_u \sqrt{f(u)/u} = f'(0).$   
 In the following sections we prove the statements made above. Our variational principle reduces to our standard variational principle 
 (see \cite{BeDe96CMP}, \cite{BeDe96PRL}, 
 \cite{BeDe98}) when $p=2$. 

\bigskip
\bigskip 
 \noindent
The rest of this manuscript is organized as follows: 
In Section 2 we derive the variational principle, 
in Section 3 the bounds for general 
values of $\gamma$ with $m>0$, $p>1$ are obtained, 
and in Section 4 we derive a lower bound of the Zeldovich-Frank-Kamenetskii type for any $\gamma \ge 0$.

\section{\, Variational Principle}

We consider left travelling wave solutions $u(\xi )$ with $\xi=x+c t$  so that the TW profile satisfies $u_{\xi} >0.$  The TW solution 
satisfies the ordinary differential equation (ODE) 
\begin{equation}\label{ode}
c u_{\xi} =  m^{p-1} \frac{d}{d\xi} \left ( u^{(m-1)(p-1)} (u_{\xi})^{p-1} \right) + f(u).
\end{equation}
From here on we denote $u' = u_{\xi}$. Following the usual procedure, we introduce the phase space coordinate 
$$
q(u) =  u^{m-1} u'(u)
$$
in terms of which the ODE for the travelling waves becomes, after dividing by $q,$ 
\begin{equation}\label{phaseeq}
\frac{c}{m^{p-1}} =  \frac{d}{d u} (q(u))^{p-1}+ \frac{ u^{m-1}f(u)}{ m^{p-1}q(u)}.
\end{equation}
Here, it is convenient to define
\begin{equation}\label{F}
F(u) = \frac{ u^{m-1}f(u)}{ m^{p-1}}.
\end{equation}

\noindent
In what follows, let us define the functional
\begin{equation}
\mathcal{J}[g] \equiv  \frac{ p\, m^{p-1} }{(p-1)^{(p-1)/p} } \frac{ \int_0^1    h(u)^{1/p}  F(u)^{(p-1)/p}  g^{(p-1)/p}  du}{\int_0^1 g(u) du}, 
\label{functional}
\end{equation}
which acts on $\mathcal{D}$, the space of functions $g \in C^1(0,1)$ such that $g(u) \ge 0$, with $h(u) = - g'(u) >0$ in $(0,1)$ and $
\int_0^1 g(u) \, du$ finite. Here
 the function $F(u)$ is given by (\ref{F}) above. 
With this notation we state our main result, which is embodied in the following theorem. 

\begin{theorem}[Variational characterization of $c_*$]
Let $f \in C^1[0,1]$ with $f(0)=f(1)=0$, $f(u) >0$ in $(0,1)$, and $f(u)$ concave in $[0,1]$. 
Assume $\gamma=m(p-1)-1 \ge  0$. 
Then, 
\begin{equation}
c_*(m,p) = J \equiv \sup \{ \mathcal{J}[g] \bigm| g \in \mathcal{D} \}
\label{eq:vc1}
\end{equation}
Moreover, 

\bigskip
\noindent
{i)} If $\gamma> 0$, there is a $g\in \mathcal{D}$, $\tilde g$ say, such that $J=\mathcal{J}[\tilde g]$. This maximizing $\tilde g$ is 
unique up to a multiplicative constant, and 

\bigskip
\noindent
{ii)} If $\gamma=0$ we construct  and explicit maximizing sequence $g_{\alpha} \in \mathcal{D}$ such that 
$\lim_{\alpha\to 0}\mathcal{J}[g_{\alpha}] = 
c_*(m,p)|_{\gamma=0}$, where $c_*(m,p)|_{\gamma=0}$ is given by (\ref{c0}) above. 
 \end{theorem}

\begin{proof}
\, Let $g(u) \in \mathcal{D}$. Multiplying (\ref{phaseeq}) by $g(u)$ and integrating  in $u$ between $0$ and $1$ we obtain after 
integrating by parts,
\begin{equation}\label{g1}
\frac{c}{m^{p-1}}\int_0^1 g(u) du  = \int_0^1 du \left( h(u) q(u)^{p-1} + \frac{g(u) F(u)}{q(u)}\right) 
 \equiv \int_0^1 \Phi(u) du.
\end{equation}
where  $h(u) \equiv -g'(u) >0$ and we assume that $g(u)$ is such that \newline $\lim_{u\rightarrow 0} g(u) q(u)^{p-1} =0$.

The integrand of the right side,
\begin{equation}
\Phi = h(u) q(u)^{p-1} + \frac{g(u) F(u)}{q(u)},
\label{DefinitionPhi}
\end{equation}
at fixed $u$ can be considered as a function of $q$. It is clear from (\ref{DefinitionPhi}) that  $\Phi(q)$ has a unique positive  minimum 
at $\hat q$ so that $\Phi(q) \ge \Phi(\hat q)$. A simple calculation yields
\begin{equation}
\hat q =\left[ \frac{ F g}{(p-1)  h } \right]^{1/p},
\label{CaseEquality}
\end{equation}
and 
$$
\Phi(\hat q) = \frac{p  g h^{1/p} F^{(p-1)/p}}{(p-1)^{(p-1)/p}}.
$$
It follows from (\ref{g1}) that
\begin{equation}
c_* \ge \frac{ p\, m^{p-1} }{(p-1)^{(p-1)/p} } \frac{ \int_0^1    h^{1/p}  F^{(p-1)/p}  g^{(p-1)/p}  du}{\int_0^1 g(u) du}, 
\label{lowerbound}
\end{equation}
for every $g \in \mathcal{D}$. To establish (\ref{eq:vc1}) we need only prove that the supremum of the right side of (\ref{lowerbound}) 
over all $g \in \mathcal{D}$ is actually 
$c_*$. we will do this separately in the cases $\gamma>0$ and $\gamma=0$.

\bigskip
\noindent
{\bf i)} {\bf Case $\gamma>0$.}  Below we show that  when $\hat q$ is the solution of (\ref{phaseeq}) (with $c=c_*$), equality is 
attained in (\ref{lowerbound}) for some 
$g \in \mathcal{D}$, 
so that we obtain the variational characterization for the speed
\begin{equation}\label{vp}
c_* = \sup_g    \frac{ p\, m^{p-1} }{(p-1)^{(p-1)/p} } \frac{ \int_0^1    h^{1/p}  F^{(p-1)/p}  g^{(p-1)/p}  du}{\int_0^1 g(u) du}.
\end{equation}
We have already proven (see (\ref{lowerbound}) above) that $c_*\ge \mathcal{J}[g]$ for every $g \in \mathcal{D}$. What we will 
actually show here is that when $\gamma>0$, there exists a $g \in \mathcal{D}$, $\tilde g$ say, such that
$c_*=\mathcal{J}[\tilde  g]$. Hence in the case $\gamma>0$ the variational principle reads, 
\begin{equation}
c_*= \max_{g \in \mathcal{D}}(\mathcal{J}[g]).
\label{newvp2}
\end{equation}
In the case $\gamma>0$, the existence of a travelling wave for any $c \ge c_*$ was proven in Theorem 2.1 of Reference 
\cite{Vasquez}. Moreover, in the case $\gamma>0$,
the solution of (\ref{phaseeq}) satisfies, 
\begin{equation}
q(u)^{p-1} \approx \frac{c_*}{m^{p-1}} u, 
\label{eq:em1}
\end{equation}
in the neighborhood of $u=0$. In order to show that the $\sup$ is actually  attained in (\ref{eq:vc1}) we have to show that there exists $
\tilde g \in \mathcal{D}$ satisfying
(\ref{CaseEquality}) when $\hat q$ is a solution of (\ref{phaseeq}). To construct such a $g$, let $v$ be the solution of 
\begin{equation}
\frac{v'}{v} =  \frac{c_*}{m^{p-1}} \frac{1}{q^{p-1}(u)},  
\label{eq:em2}
\end{equation}
where $q$ is a solution of (\ref{phaseeq}). Notice that this $v$ is unique up to a multiplicative constant. A simple calculation using 
(\ref{eq:em2}), (\ref{phaseeq}), and the definition (\ref{F}) of $F$, yields, 
\begin{equation}
\frac{v''}{v} =  \frac{c_*}{m^{p-1}} \frac{F(u)}{q^{2 p-1}(u)}. 
\label{eq:em3}
\end{equation}
Choosing
\begin{equation}
\tilde g(u) = \frac{1}{(v'(u))^{1/(p-1)}}, 
\label{eq:em4}
\end{equation}
it follows from (\ref{eq:em2}) and  (\ref{eq:em3}) that, 
\begin{eqnarray}
&-\tilde g'(u) = \frac{1}{p-1}  \frac{1}{(v'(u))^{p/(p-1)}} \, v''= \frac{1}{p-1}  \frac{1}{(v'(u))^{/(p-1)}} \, \frac{v''}{v} \frac{v}{v'} \nonumber
\\
&= \frac{1}{p-1}  \, g(u) \, \frac{F(u) }{q^{p}(u)}.
\label{eq:em5}
\end{eqnarray}
which is precisely (\ref{CaseEquality}). From (\ref{eq:em1}) and (\ref{eq:em2}) we have that 
\begin{equation}
v(u) \approx A \, u \qquad \mbox{and} \qquad v'(u) \approx A,
\label{eq:em6}
\end{equation}
near $u=0$. Hence, it follows from (\ref{eq:em4}) that 
\begin{equation}
\tilde g(0) = A^{-1/(p-1)} < \infty. 
\label{eq:em7}
\end{equation}
Integrating (\ref{eq:em2}) and using (\ref{eq:em6}) we can write explicitly, 
\begin{equation}
v(u) = {\rm exp} \left(  \int_{u_0}^u  \frac{c_*}{m^{p-1}} \frac{1}{q^{p-1}(s)} \, ds \right),
\label{eq:em8}
\end{equation}
for some $0 < u_0<1$. Clearly, the value of $A$ in (\ref{eq:em6}) is determined by the value of $u_0$. Finally, 
using  (\ref{eq:em2}), (\ref{eq:em4}), and  (\ref{eq:em8}), we 
can write, 
\begin{equation}
\tilde g(u) =      \frac{m \, q(u)}{{c_*}^{1/(p-1)}}  \,  {\rm exp} \left( \frac{1}{p-1} \,  \int_{u}^{u_0}  \frac{c_*}{m^{p-1}} 
\frac{1}{q^{p-1}(s)} \, ds \right).
\label{eq:em9}
\end{equation}
Since, the integrand in (\ref{eq:em9}) is positive, $u_0 <1$, and $q(1)=0$, it follows from  (\ref{eq:em9}) that $\tilde g(1)=0$. From all 
the results above it follows that 
$\tilde g$ given by (\ref{eq:em9}) is in $\mathcal{D}$, and that  $c_* = \mathcal{J}[\tilde g]$.   

\bigskip
It is clear from the construction above that $\tilde g$ is unique up to a multiplicative cosntant. The uniqueness of the maximizing $g \in 
\mathcal{D}$, however, can be 
seen directly from our variational principle (\ref{functional}). In fact, suppose that there are two different maximizers, say $g_1, g_2  \in 
\mathcal{D}$, with 
$\int_0^1 g_1(u) \, du = \int_0^1 g_2(u) \, du =1$. Then, for any $\alpha \in (0,1)$ consider now, 
\begin{equation}
g_{\alpha} (u) = \alpha \,  g_1(u) + (1-\alpha) \, g_2(u).
\label{eq:em10}
\end{equation}
It is clear from (\ref{eq:em10}) that $g_{\alpha} \in \mathcal{D}$ and that $\int_0^1 g_{\alpha} (u) \, du =1$. Using H\"older's inequality 
with exponents $p$ and $p'=p/(p-1)$, 
it follows from (\ref{functional}) that 
\begin{equation}
\mathcal{J} [g_{\alpha}] >   \alpha \,  \mathcal{J}   [g_1] + (1-\alpha) \, \mathcal{J}  [g_2] = c_*, 
\label{eq:em11}
\end{equation}
which is a contradiction with the fact that $g_1$ and $g_2$ are the maximizers. Notice that the inequality in (\ref{eq:em11}) is strict if 
$g_1 \not \equiv g_2$.  

\bigskip
\bigskip
\noindent
{\bf ii)} {\bf Case $\gamma=0$.} For later purposes it is convenient to denote
\begin{equation}
J_g[f] =  \int_0^1 [ u^{m-1} h(u)^m  f(u) g(u)]^{1/(m+1)} \, du.
\label{jotas}
\end{equation}
It then follows from (\ref{F}) and (\ref{functional}) that 
\begin{equation}
\mathcal{J}[g] = p \,  m^{2/(m+1)} \,  J_g[f] 
\label{eq:J1}
\end{equation}
in the case $\gamma=0$, when we conveniently normalize $g$ so that $\int_0^1 g(u) \, du =1$. 
Now, choose as a trial function the sequence
\begin{equation}\label{galfa}
g_\alpha(u) = \frac{\alpha}{1-\alpha}  (u^{\alpha-1} - 1), \qquad 0 < \alpha <1, \qquad \mbox{with}  \qquad \alpha \to 0.
\end{equation}
Notice that for each $\alpha \in (0,1)$, $g_{\alpha}(u) >0$, $g_{\alpha}'(u) <0$, $g_{\alpha}(1)=0$, and $\lim_{u \to 0}  [u \, g_{\alpha}
(u)] = 0$, so these are appropriate trial functions. 
Moreover, we have normalized the $g_{\alpha}$'s so that $\int_0^1 g_{\alpha}(u) \, du = 1$.

\bigskip

With this choice we will show that $\lim_{\alpha \rightarrow 0} J_{g_\alpha} [f] = f'(0)^{1/(m+1)} $ so that 
$\mathcal{J}[g_{\alpha}] \to ( m^2 p^{m+1} f'(0))^{1/(m+1)} = c_0(m,p)$ as $\alpha \to 0$. To do so we write
$J[f] = J[u f'(0)] + J[f]- J[u f'(0)]$ and show that 
\begin{equation}
J_{g_\alpha}[u f'(0)] \to  f'(0)^{1/(m+1)}, \qquad J_{g_\alpha}[f]- J_{g_\alpha}[u f'(0)] \to 0, \qquad \mbox{as} \qquad  \alpha \to 0.
\label{eq:J2}
\end{equation}
While the proof of the second limit is given in the Appendix, the proof of the first is as follows. Using (\ref{jotas}) with $g=g_{\alpha}$ we have, 
\begin{eqnarray}
J_{g_\alpha}[u f'(0)] &= f'(0)^{1/(m+1)} \alpha (1-\alpha)^{-1/(m+1)}\int_0^1  ( u^{m (\alpha-1) } ( u^{\alpha-1} -1) )^{1/(m+1)} \, du
\nonumber \\ 
&= f'(0)^{1/(m+1)}  \alpha (1-\alpha)^{-(m+2)/(m+1)} B\left( \frac{m+2}{m},\frac{\alpha}{1-\alpha}\right),
\end{eqnarray}
where $B(x,y)$ denotes the Euler Beta function.
Now, $B(t,s) = \Gamma(t) \Gamma(s)/ \Gamma(t+s)$, hence 
\begin{equation}
J_{g_\alpha}[u f'(0)] = f'(0)^{1/(m+1)}  \alpha (1-\alpha)^{-(m+2)/(m+1)} \frac{ \Gamma(\frac{m+2}{m}) \Gamma( \frac{\alpha}{1-\alpha})} 
{\Gamma( \frac{m+2}{m} + \frac{\alpha}{1-\alpha})} 
\label{eq:J3}
\end{equation}
Using  $\lim_{x \to 0} x \, \Gamma(x)=1$  to evaluate the limit of the right side of (\ref{eq:J3}) when $\alpha \to 0$, we finally conclude, 
$J_{g_\alpha}[u f'(0)] \to  f'(0)^{1/(m+1)}$ as $\alpha \to 0$ from above. As indicated before, (\ref{eq:J2}) then implies that 
$\mathcal{J}[g_{\alpha}] \to ( m^2 p^{m+1} f'(0))^{1/(m+1)} = c_0(m,p)$ as $\alpha \to 0$, which concludes the proof of the Theorem. 
\end{proof}

\section{An upper bound on the speed  for $\gamma \ge  0$}

In this section we derive from our variational principle (i.e., from Theorem 1 above) an explicit upper bound on the speed of fronts. 
In order to do this we rewrite (\ref{vp}) as
\begin{equation} 
c_* =    \sup_g  \left[  \frac{ p\, m^{p-1} }{(p-1)^{(p-1)/p} } \frac{ \int_0^1    [  h  F^{(p-1)} /g ] ^{1/p}  g   du}{\int_0^1 g(u) du}
\right].
\end{equation}
Since the mapping $x \to x^{1/p}$ is concave for $p>1$, defining the probability measure  $d\nu = g(u) du/\int_0^1 g du$,  and using 
Jensen's inequality we get, 
\begin{eqnarray} \label{Jen1}
& c_* \le  \sup_g   \frac{ p\, m^{p-1} }{(p-1)^{(p-1)/p} }  \left[ \frac{ \int_0^1     h  F^{(p-1)}    du}{\int_0^1 g(u) du}\right] ^{1/p} \nonumber\\
 \le & \frac{ p\, m^{p-1} }{(p-1)^{(p-1)/p} }  \left[ \sup_u\left(\frac{F^{p-1}}{u}\right)  \sup_g \frac{ \int_0^1 h(u)\,   u \,du}{\int_0^1 g(u) du}
 \right] ^{1/p}.
\end{eqnarray}
Integrating $\int_0^1 h(u)\, u \, du$ by parts, using $g(1) =0$, and $\lim_{u\rightarrow 0} u g(u) =0 $ it follows from (\ref{Jen1}) that
$$
c_* \le   \frac{ p\, m^{p-1} }{(p-1)^{(p-1)/p} }   \left[ \sup_u \frac{F^{p-1}}{u}\right]^{1/p}.
$$
Replacing the expression for $F(u) $ in terms of $f(u)$  we finally obtain the  upper bound
\begin{equation}
c_* \le    p \left(  \frac{m}{p-1} \right)^{(p-1)/p}  \sup_u \left[  u^{\gamma} \left(  \frac{f(u)}{u}\right)^{(p-1)}\right]^{1/p}
\label{upga}
\end{equation}
with $\gamma = m(p-1)-1$ as defined before. When $\gamma=0$ the expression above reduces to 
\begin{equation}\label{upg0}
c_*|_{\gamma=0} \le  p (m^2)^{1/(m+1)} \sup_u  \left(  \frac{f(u)}{u}\right)^{1/(m+1)}.
\end{equation}
In particular, when $m=1$ (i.e., $p=2$ since $\gamma=0$), (\ref{upg0}) is the classical upper bound of Aronson and Weinberger 
\cite{AW78}. 
Notice that for the reaction profiles considered here (i.e., $f(u)$ positive and concave in $[0,1]$, $f(0)=f(1)=0$ and $f \in C^1[0,1]$, we 
clearly have that
$\sup_{u\in [0,1]} f(u)/u = f'(0)$, and in fact we have equality in (\ref{upg0}). 

\section{\, Integral lower bound: a Zeldovich--Frank--Kamenetskii type bound}

From the variational characterization lower bounds can be constructed choosing specific values for the trial function $g(u)$. 
In this section we construct a lower bound which involves the integrals of the reaction term as   the Zeldovich--Frank--Kamenetskii
classical bound \cite{ZFK38,BeNi92,BeDe98}. Our  ZFK type bound is embodied in the following lemma. 

\begin{lemma}
For any $m>0$,  $p>1$, $\gamma \ge 0$ and $f$ satisfying the hypothesis of Theorem 1,  we have that
\begin{equation}\label{ZFKnew}
c_* \ge  \left(  \frac{ m p }{p-1} \right)^{(p-1)/p}  \left[  \int_0^1  u^{m-1} f(u)  du\right] ^{(p-1)/p}.
\end{equation}
\end{lemma}

\begin{proof}

Choose as a trial function of our variational principle (\ref{eq:vc1}) the function
$$
g(u) = \left( \int_u^1  F(u')   du' \right)^{1/p}.
$$
It is simple to verify that  $g \in C^1(0,1)$, $h = - g' >0$,  $g(1) =0$, and $g(u) \ge 0$ in $[0,1]$. Moreover, since $g$ is decreasing, 
$\int_0^1 g(u) du \le g(0) =  \left( \int_0^1  F(u')   du' \right)^{1/p} < \infty$. Hence, $g \in \mathcal{D}$. 
A simple calculation yields
$$
h(u) = \frac{F(u)}{p}  \left( \int_u^1  F(u')   du' \right)^{1/p-1}
$$
and $ h g^{p-1} = F(u)/p$. 
It follows then from (\ref{vp}) that
$$
c_* \ge  \frac{ p\, m^{p-1} }{(p-1)^{(p-1)/p} } \left( \frac{1}{p} \right)^{1/p} \frac{ \int_0^1   F(u)  du}{\int_0^1 g(u) du}.
$$
Now, since $g(u)$ is a decreasing positive function, $\int_0^1 g(u) du \le g(0)$. Hence, 
\begin{equation}
c_* \ge  \frac{ p\, m^{p-1} }{(p-1)^{(p-1)/p} } \left( \frac{1}{p} \right)^{1/p}  \left( \int_0^1   F(u)  du\right)^{(p-1)/p}.
\label{ZFK1}
\end{equation}
If we express the right side  of (\ref{ZFK1}) in terms of the original reaction term $f(u)$ we get (\ref{ZFKnew}) which proves the lemma.
\end{proof}


\begin{acknowledgements}
This work was supported by Fondecyt (Chile) projects 114--1155, 116--0856 and 
by Iniciativa Cient\'\i fica Milenio, ICM (Chile), through the Millennium Nucleus RC--120002. 
\end{acknowledgements}



\section*{Appendix}

In this appendix we show that 
\begin{equation}
J_{g_{\alpha}}[f]- J_{g_{\alpha}}[u f'(0)]\rightarrow 0 \qquad \mbox{when}    \qquad \alpha\rightarrow 0,
\end{equation}
where $g_{\alpha}$ is given by  (\ref{galfa}). 
We defined
$$
J_{g_{\alpha}}[f] =  \int_0^1 [ u^{m-1} h_{\alpha}^m  f(u) g_{\alpha}(u)]^{1/(m+1)} \, du,
$$
so that
\begin{eqnarray}
& \left| J_{g_{\alpha}}[f] - J_{g_{\alpha}}[u f'(0)] \right| \le  \nonumber \\ 
&\int_0^1 ( u^{m-1} h_\alpha^m \, g_{\alpha}(u))^{1/(m+1)} \, | f(u)^{1/(m+1)} -   (u f'(0))^{1/(m+1)}  | \,  du.
\label{A0}
\end{eqnarray}
Since for $m \ge 0$, $1/(m+1) \le 1$, it is not difficult to verify the inequality
$
| a^{1/(m+1)} - b^{1/(m+1)}| \le |a-b|^{1/(m+1)} 
$
for all $a\ge 0, b \ge 0, m\ge 0. $ In the present   case, we have
\begin{equation} \label{A1}
| f(u)^{1/(m+1)} -   (u f'(0))^{1/(m+1)}  |  \le  | f(u) - u f'(0)| ^{1/(m+1)} .
\end{equation}
If $f(u)$ and its derivative are continuous in $[0,1]$, there exist $d>0,  k>0$ such that
\begin{equation} \label{A2}
\frac{| f(u) - u f'(0)|}{u} < d \, \,  u^k.
\end{equation}
Using (\ref{A1}) and (\ref{A2}) in (\ref{A0}), together with the explicit form of $g_{\alpha}$ we have that
$$
\left| J_{g_{\alpha}}[f] - J_{g_{\alpha}}[u f'(0)] \right| \le \frac{\alpha}{(1 -\alpha)^{1/(m+1)}} \int_0^1 d^{1/(m+1)} u^{N(\alpha)} \, du,
$$
where
$$
N(\alpha) = \alpha - 1 + \frac{k}{(m+1)}.
$$
Since $\alpha>0$ and $k>0$, $N(\alpha) > -1$ and $u^{N(\alpha)}$ is integrable. Performing the integral we finally find
$$
\left| J_{g_{\alpha}}[f] - J_{g_{\alpha}}[u f'(0)] \right| \le \frac{m+1} { \alpha (m+1) + k} \, \, \frac{ \alpha d^{1/(m+1)}}{ (1-\alpha)^{1/(m
+1)}}  \rightarrow 0 \qquad \mbox{when} \qquad \alpha\rightarrow 0.
$$

\end{document}